\documentclass[12pt]{article}
\usepackage{authblk}
\usepackage{hyperref}
\usepackage{amssymb}
\usepackage{amsthm}
\usepackage{amsmath}
\usepackage{multirow}
\usepackage{float}
\usepackage{graphicx}
\usepackage{listings}
\usepackage{MnSymbol}
\usepackage{titlesec}
\usepackage{csquotes}
\newtheorem{theorem}{Theorem}[section]

\newtheorem{proposition}[theorem]{Proposition}

\newtheorem{definition}[theorem]{Definition}
\newtheorem{example}[theorem]{Example}

\newtheorem{remark}[theorem]{Remark}

\titleclass{\subsubsubsection}{straight}[\subsubsection]

\newcounter{subsubsubsection}[subsubsection]
\renewcommand\thesubsubsubsection{\thesubsubsection.\arabic{subsubsubsection}}

\titleformat{\subsubsubsection}
  {\normalfont\normalsize\bfseries}
  {\thesubsubsubsection}{1em}{}

\titlespacing*{\subsubsubsection}
  {0pt}{3.25ex plus 1ex minus .2ex}{1.5ex plus .2ex}

\newcounter{subsubsubsubsection}[subsubsubsection]

\renewcommand\thesubsubsubsubsection
  {\thesubsubsubsection.\arabic{subsubsubsubsection}}

\titleformat{\subsubsubsubsection}
  {\normalfont\normalsize\bfseries}
  {\thesubsubsubsubsection}{1em}{}

\titlespacing*{\subsubsubsubsection}
  {0pt}{3.25ex plus 1ex minus .2ex}{1.5ex plus .2ex}
\usepackage{tikz}
\usetikzlibrary{trees}
\begin{document}
\title{\Large\bfseries Linearization Problem for a System of \\Two Second-Order ODEs via Cartan's \\Method: Branch I}
\author[1]{Batoul M. Raddad\thanks{1227002@student.birzeit.edu}}
\author[1]{Ahmad Y. Al-Dweik\thanks{aaldweik@birzeit.edu}}
\author[1]{Marwan Aloqeili\thanks{maloqeili@birzeit.edu}}
\author[2]{F. M. Mahomed\thanks{Corresponding Author: Fazal.Mahomed@wits.ac.za}}
\affil[1]{Department of Mathematics, Birzeit University, Ramallah, Palestine}
\affil[2]{School of Computer Science and Applied Mathematics, 
University of the Witwatersrand, Johannesburg, Wits 2050,  South Africa}
\maketitle
\begin{abstract}
Cartan's method classifies the class of linearizable system of two second-order ODEs into many branches. This paper investigates Branch I of the classification, characterized by a rank-one generalized Wilczynski  invariant matrix and the vanishing of two relative invariants $K_1$ and $L_1$. It is demonstrated that any linearizable system belonging to this branch admits an eight-dimensional Lie point symmetry algebra. The canonical form for this class is provided and the invariant characterizations based on the obtained rank-zero invariant coframe and the corresponding constant structure equations are established. Also, a systematic procedure for constructing the linearizing point transformation is derived. The theoretical results are illustrated by several examples.	
\end{abstract}

\section{Introduction and Literature Review}
 The symmetry structures for systems of two second-order ordinary differential equations (ODEs)  are of significant interest.  Gorringe and Leach \cite{p13} 
showed that every linear system of two second-order ODEs with constant coefficients of the form $\textbf{u}^{\prime\prime}=A\textbf{u}$ can be transformed into an equivalent system of the form $\textbf{u}^{\prime\prime}=M\textbf{u}$, where $M$ is a constant upper triangular matrix. Moreover, the  dimension of the corresponding Lie symmetry algebra is determined by the coefficient matrix $M$. In particular, such systems admit seven, eight or fifteen Lie point symmetries.
In \cite{R1}, Wafo Soh and Mahomed demonstrated that a linear system of two second-order ODEs may admit five, six, seven, eight or fifteen point symmetries. They showed that any linear system of two second-order ODEs can be reduced to the free particle equations 
\begin{equation}\label{two}
u''_i=0,~~~i=1,2
\end{equation}
 if and only if it admits the maximal number of symmetries. Moreover, a system of two second-order ODEs admitting five or six Lie point symmetries cannot be transformed into a system with constant coefficients. In contrast, a system with either constant or variable coefficients may admit seven or eight Lie point symmetries. Wafo Soh \cite{p37} studied the symmetry structure
of a system of $n\geq2$ linear second-order ODEs with constant coefficients using Jordan canonical forms and appropriate scaling. He proved that any system of two linear second-order ODEs with complex constant coefficients admits a seven-, eight- or fifteen-dimensional Lie symmetry algebra. In addition, he established three canonical forms for such systems. In \cite{p30}, Meleshko noted that Wafo Soh \cite{p37} had previously studied systems of linear second-order ODEs with constant coefficients but in cases where the coefficient matrix has only real eigenvalues. In contrast, he mentioned that to complete the analysis, one needs to include the cases involving complex eigenvalues corresponding to diagonal Jordan blocks. In \cite{p33}, Meleshko et al. mentioned that any linear system of two second-order ODEs of the form $\textbf{u}^{\prime\prime}=A\textbf{u}^{\prime}+B\textbf{u}+f$,
where $A$ and $B$ are constant matrices that commute, can be reduced to the simpler form $\textbf{u}^{\prime\prime}=M\textbf{u}$, with $M$ being a constant matrix. However, they extended this work by completing the group classification where the commutative property does not hold. In this case, the system can be reduced to the form
$\textbf{u}^{\prime\prime}=M(x)\textbf{u}$, with $M(x)$ depending on the independent variable. Bagderina \cite{p20} classified linear system of two 
second-order ODEs into five types and provided the basis of invariants and invariant differentiation operators using Lie's method. These invariants offer a straightforward method for identifying equations that may be equivalent to a given system, as well as determining the transformations that relate two equivalent systems. In \cite{R2}, Bagderina extended her work and provided a symmetry characterization of an arbitrary system of two homogeneous linear second-order ODEs using invariant theory. This invariant-based classification yields eleven non-equivalent symmetry types. However, the absence of explicit canonical forms for the classes of linear systems remains a major challenge.\\

The linearization problem for an ODE (or a system of ODEs) under point transformations consists of establishing necessary and sufficient conditions for local equivalence to a linear ODE (or system of ODEs) and constructing the corresponding linearizing point transformation. Algebraic linearization criteria for dimension four algebras for systems of two second-order ODEs were studied in the works Ayub et al., and Wafo and Mahomed \cite{Ay,WM}. Extensions of Lie algebraic criteria for linearizability of system of second-order ODEs are also given in \cite{Ay}. The linearization problem for a system of $n$ second-order ODEs has been extensively studied for many years and continues to attract significant attention. Chern \cite{R3} employed Cartan's equivalence method to analyze the equivalence of systems of $n$ second-order ODEs under invertible Lie transformations of the form $\bar{t}=t,~ \bar{x}^{i}=\psi^{i}(x^{j})$ and $\bar{t}=t$, $\bar{x}^{i}=\psi^{i}(t, x^{j})$. Fels \cite{p15} applied Cartan's equivalence method to study the equivalence of the system of $n$ second-order ODEs
\begin{equation}\label{nsys}
\mathbf{u}'' = \mathbf{f}(x,\mathbf{u},\mathbf{p}), 
\qquad \mathbf{u}=(u_1,u_2,\dots,u_n), \qquad  \mathbf{p}=(p_1,p_2,\dots,p_n),
\end{equation}
where $p_i$ denotes $u'_i$ and $\mathbf{f}=(f^1,f^2,\dots,f^n)$ is a smooth 
vector-valued function, under point transformation to the free particle system 
\begin{equation}
\mathbf{u}'' =0.
\end{equation}
He found that the associated structure group has dimension $n^2+4n+3$. He proved the following theorem.
\begin{theorem}\cite{p15}
There exists a unique $\{e\}-$structure with a maximal dimensional symmetry group. For this $\{e\}-$structure the structure functions vanish, and a representative for the system of equations (\ref{nsys}) giving rise to this $\{e\}-$structure is the system of free particle equations 
\begin{equation}\label{m}
u''_i=0,~~~i=1, ..., n.
\end{equation}
The equivalence class of these equations is invariantly characterized by the two vanishing conditions
\begin{equation}\label{Fels}
\begin{aligned}
&(\tilde{P}^{i}_{j})|_{e}=\frac{1}{2}\frac{d}{dx}f^{i}_{|j}-f^{i}_{,j}-\frac{1}{4}f^{i}_{|k}f^{k}_{|j}-\frac{1}{n}\delta^{i}_{j}(\frac{1}{2}\frac{d}{dx}f^{k}_{|k}-f^{k}_{,k}-\frac{1}{4}f^{l}_{|k}f^{k}_{|l})=0,\\
&(\tilde{S}^{i}_{jkl})_{|e}=f^{i}_{|jkl}-\frac{3}{n+2}f^{m}_{|m(jk}\delta^{i}_{l)}=0, ~~~ where~~~ f^{i}_{,j}=\frac{\partial f^{i}}{\partial u_{j}} ~~and~~ f^{i}_{|j}=\frac{\partial f^{i}}{\partial p_{ j}}.\\
\end{aligned}
\end{equation}\end{theorem}
Wilczynski \cite{26} introduced the Wilczynski invariants as the fundamental set of invariants for linear scalar ODEs under the action of the Lie pseudogroup
\begin{equation}
(x,y)\rightarrow(\lambda(x),\eta(x)y).
\end{equation} They form a complete and minimal set of relative invariants from which all other invariants can be derived by invariant differentiation.
These invariants were later extended to any system of linear ODEs by Se-ashi \cite{20}.
Boris \cite{21} introduced the generalized Wilczynski invariants evaluated at any linearizable system of ODEs. Associated with system (\ref{nsys}), the generalized Wilczynski invariant matrix $W=(W^i_j)$ is defined by the trace-free matrix
\begin{equation}
W
=
\Phi - \frac{1}{2}(\mathrm{tr}\,\Phi)\,I,
\end{equation}
where  
\begin{equation}
\Phi = A - \frac{1}{2}D_x B +\frac{1}{4}B^2,
\end{equation}
and
\begin{equation}
A^i_j = \frac{\partial f^i}{\partial u_j}, 
\qquad 
B^i_j = \frac{\partial f^i}{\partial p_j},
\end{equation}
and $D_x$ is the total derivative $D_x = \frac{\partial}{\partial x}
+ p_k\frac{\partial}{\partial u_k}
+ f^k\frac{\partial}{\partial p_k}.$ He mentioned that the generalized Wilczynski invariant matrix vanishes for the free particle equations (\ref{m}) and therefore any equivalent system.
Sakka \cite{p18} demonstrated that any linearizable system of two second-order ODEs must take the form
 \begin{equation}\label{cubic}
\begin{aligned}
& u_{1}^{\prime\prime}+a_{11}{u'_1}^{3}+a_{12}{u'_1}^{2}u^{\prime}_{2}+a_{13}u^{\prime}_{1}{u'_2}^{2}+a_{14}{u'_1}^{2}+a_{15}u^{\prime}_{1}u^{\prime}_{2}+a_{16}{u'_2}^{2}\\
&+a_{17}u^{\prime}_{1}+a_{18}u^{\prime}_{2}+a_{19}=0,\\
&u_{2}^{\prime\prime}+a_{13}{u'_2}^{3}+a_{12}{u'_2}^{2}u^{\prime}_{1}+a_{11}u^{\prime}_{2}{u'_1}^{2}+a_{24}{u'_1}^{2}+a_{25}u^{\prime}_{2}u^{\prime}_{1}+a_{26}{u'_2}^{2}\\
&+a_{27}u^{\prime}_{1}+a_{28}u^{\prime}_{2}+a_{29}=0,\\
\end{aligned}
\end{equation} 
where the coefficients $a_{ij}$ depend on $x, u_{1}$ and $u_{2}$. He addressed the linearization problem for such systems using Lie’s method by identifying the equivalence group of transformations associated with system (\ref{cubic}) and deriving the corresponding invariant equations. By solving these equations, he determined the first and second-order relative invariants of the transformations, which serve as necessary conditions but not necessarily sufficient for linearizability. Mahomed and Qadir \cite{p26} showed that the minimum and maximum possible dimensions of the Lie symmetry algebra of a linearizable system of $n$ second-order ODEs are $2n+1$ and $n^{2}+4n+3$, respectively. They also proved that, besides the minimal dimension $2n+1$, the other possible dimensions range from $2n+2$ to $(n+2)^{2}/2$ when $n$ is even, and from $2n+2$ to $[(n+2)^{2}+1]/2$ when $n$ is odd. Furthermore, they derived necessary and sufficient conditions under which a system of two nonlinear second-order ODEs, cubic in the first derivatives and associated with the maximal dimension of the symmetry algebra, can be reduced to the free particle equations via a point transformation. They suggested that the linearization problem for systems associated with the other symmetry algebras remains unresolved. They wrote:
\begin{displayquote}
\textit{"It would be important to find ways of providing the linearizability criteria for the cases of the
other symmetry algebras".}
\end{displayquote}
Neut et al. \cite{R5} used Cartan’s method to determine the conditions under which a system of two second-order ODEs is equivalent to the free particle equations (\ref{two}). They obtained 88 fundamental invariants using Cartan’s method. By Poincaré lemma, they showed that only two of these invariants are sufficient to form a basis for the differential ideal generated by all 88 invariants where these two fundamental invariants identically vanish at the free particle equations (\ref{two}). They proved the following theorem.
\begin{theorem}\label{Nuet3} \cite{R5}
The system $u''_1=f^1(x,u_1,u_2,p_1,p_2),u''_2=f^2(x,u_1,u_2,p_1,p_2)$ is equivalent to the system $u''_1=0,u''_2=0$ under a point transformation if and only if
\begin{equation}\label{Nuet1}
\begin{aligned}
&2D_xf^1_{p_1}-f^1_{p_2}f^1_{p_1}-f^1_{p_2}f^2_{p_2}-4f^1_{u_2}=0,\\
&(f^1_{p_1})^2-(f^2_{p_2})^2+4(f^1_{u_1}-f^2_{u_2})+2D_x(f^2_{p_2}-f^1_{p_1})=0,\\
&2D_xf^2_{p_1}-f^2_{p_2}f^2_{p_1}-f^2_{p_1}f^1_{p_1}-4f^2_{u_1}=0,\\
\end{aligned}
\end{equation}
\begin{equation}\label{Nuet2}
\begin{aligned}
&f^2_{p_1p_1p_1}=0, ~~~f^1_{p_2p_2p_2}=0,~~~f^2_{p_2p_2p_2}-3f^1_{p_1p_2p_2}=0,\\
&f^1_{p_1p_1p_1}-3f^2_{p_1p_1p_2}=0,~~~f^1_{p_1p_1p_2}-f^2_{p_1p_2p_2}=0,
\end{aligned}
\end{equation}
where $D_x$ is the total derivative.\end{theorem}
 It is worth noting that the set of conditions (\ref{Nuet1}) is equivalent to the vanishing of $(\tilde{P}^{i}_{j})|_{e}$, while the set of conditions (\ref{Nuet2}) is equivalent to the vanishing of $(\tilde{S}^{i}_{jkl})_{|e}$ given by Fels \cite{p15}. Bagderina  established the general linearization criteria for a system of two 
 second-order ODEs by utilizing the direct method for the form (\ref{cubic}). She demonstrated that the linearizing transformation can be derived from the solution of a nonlinear system of 15 partial differential equations (PDEs), where some are cumbersome and need symbolic package such as Maple or Mathematica as she hinted in \cite[Remark 2]{p16}. Moreover, her work applies to the form (\ref{cubic}) without considering different canonical forms with five, six, seven, eight and fifteen dimensional Lie symmetry algebras.\\

In this work, we begin to address the problem suggested by Mahomed and Qadir by giving a linearization criteria for the other symmetry algebras. Therefore, we begin to solve the linearization problem for systems of two second-order ODEs admitting five, six, seven and eight Lie symmetry algebras, under invertible point transformations. We apply a new framework of Cartan's equivalence method by branching via relative invariants and introducing a chain of auxiliary functions to overcome the problem of swilling of expressions. This work provides a complete Cartan equivalence analysis of the linearization problem for systems of two second-order ODEs. The reduction procedure naturally gives rise to a branching classification, with each branch representing a distinct geometric class of systems. For every branch, we determine a canonical form, an invariant coframe, a complete characterization in terms of differential invariants, and a corresponding linearization theorem. Furthermore, we provide an explicit construction of the linearizing point transformation by solving a system of linear and Riccati PDEs together with illustrative examples. Since the analysis associated with each branch is substantial and largely independent of the others, we plan to present the complete classification as a series of self-contained manuscripts. Each manuscript will focus on a specific branch of the classification, while collectively they constitute a unified Cartan equivalence framework.\\

It is worth emphasizing that only one branch, namely the one in which the generalized Wilczynski invariant matrix has rank-zero and $\tilde{S}=O$, has been commonly studied in the literature. This branch characterizes the class of free particle equations. It seems that the remaining branches have not been systematically characterized within a complete Cartan equivalence framework. The present paper is devoted to the analysis of Branch I, characterized by  a rank-one generalized Wilczynski invariant matrix and the vanishing of the relative invariants $K_1,L_1$. This branch constitutes the first branch of the initial classification tree presented in Appendix A. In future work, we will investigate Branch II, which is characterized by a rank-one generalized Wilczynski invariant matrix together with the relative invariants $K_1=0,L_1\neq0$.\\

This work is organized as follows. Section 2 presents the mathematical background required for the subsequent analysis. In section 3, we apply Cartan's equivalence method to obtain an invariant characterization of the linearizable system whose associated generalized Wilczynski invariant matrix has rank-one. We consider the first branch where the relative invariants $K_1,L_1$ are identically zero, and derive the corresponding prolonged invariant coframe. This analysis yields necessary and sufficient conditions for a system of two second-order ODEs to belong to this branch. In section 4, we determine the canonical form for linear systems in this branch. Then we state the main result. Finally, in section 5, we develop a systematic procedure for constructing the linearizing map that reduces a linearizable system of two 
second-order ODEs to its canonical form and we illustrate the method with several examples.

\section{Preliminaries}
System (\ref{nsys}) naturally arises in the study of path geometries and projective differential structures and their classification under point transformations remains a central problem in the geometric theory of differential equations.\\

Under the invertible point transformation
\begin{equation}
\bar{x}=\xi(x,\mathbf{u}),\qquad 
\bar{\mathbf{u}}=\phi(x,\mathbf{u}),
\end{equation}
the generalized Wilczynski invariant matrix $W$ transforms into $\overline{\vphantom{x}W}$ as follows:
\begin{equation}
\overline{\vphantom{x}W}
=
\left(D_x \xi\right)^{-2} J W J^{-1},
\qquad
J=\phi_{\mathbf{u}}-\left(D_x \xi\right)^{-1}(D_x \phi)(\xi_{\mathbf{u}}),
\end{equation}
 where $$\phi=(\phi_1,\ldots,\phi_n)^T,~\phi_{\mathbf{u}}
\in GL(n,\mathbb{R}),~D_x \phi\in \mathbb{R}^{n\times1}, \xi_{\mathbf{u}}
\in \mathbb{R}^{1\times n} .$$
Consequently, the rank  of $W$ is invariant and the determinant is relative invariant under a point transformation. 
\begin{definition}
The system of second-order ODEs is called \emph{generic} if
\begin{equation}
\det W \neq 0,
\end{equation}
and \emph{non-generic} if
\begin{equation}
\det W = 0.
\end{equation}
\end{definition}

In the generic case, the generalized Wilczynski invariant matrix has full rank and defines a non-degenerate projective curvature endomorphism. In contrast, the 
non-generic case is characterized by a reduction of curvature rank, equivalently the existence of a nontrivial invariant direction in the associated Cartan geometry. This degenerate regime includes both the flat case $W=O$ and the intermediate case $W\neq O$ with $\det W=0$, which plays a distinguished role in symmetry enhancement and geometric reduction.\\

 In our work we restrict our attention to the case $W\neq O$  (equivalently rank$~W> 0$). In this case, $W$ is nonscalar matrix and we need the following theorem.
\begin{theorem}\cite{p47}\label{s1-thm1}
Every nonscalar matrix $A\in \mathbb{C}^{n\times n}$ is similar to a matrix with all entries different from zero.
\end{theorem}
\begin{remark}\label{rm10}
Theorem \ref{s1-thm1} implies that any nonzero generalized Wilczynski
invariant matrix $W$ is similar to a matrix $\overline{\vphantom{x}W}$ with all entries different form zero. Hence, any system can be reduced to an equivalent one whose associated generalized Wilczynski
invariant matrix $W$ has nonzero entries. Hence, without loss of generality, we consider the case in which $W$ has nonzero entries.
\end{remark}

\begin{theorem}
The canonical Laguerre-Forsyth form of the system of linear second-order ODEs
\begin{equation}\label{131}
\mathbf{u}''+A(x)\mathbf{u}'+B(x)\mathbf{u}=0,
\end{equation}
where
$\mathbf{u}=
(u_1, u_2,\dots,u_n),
A(x),B(x)\in M_{n}(\mathbb{C})$, is 
\begin{equation}
\mathbf{U}''+K(x)\mathbf{U}=0, ~~~\text{tr}(K)=0
\end{equation}
and
$\mathbf{U}=
(U_1, U_2,\dots,U_n),
K(x)\in M_{n}(\mathbb{C})$. 
\end{theorem}
\begin{proof}
Consider the linear system (\ref{131}). Then the transformation
\begin{equation}\label{130}
\mathbf{u}=P(x)\mathbf{z},
\end{equation}
where $P(x)$ is an invertible matrix, satisfies
\begin{equation}
2P'+AP=0,
\end{equation}
  reduces system (\ref{131}) to 
\begin{equation}\label{132}
\mathbf{z}''+C(x)\mathbf{z}=0,
\end{equation}
where $C(x)$ is given by
\begin{equation}
C=
P^{-1}
\left(
B-\frac12A'
-\frac14A^2
\right)P.
\end{equation}
Now, we introduce a change of independent variable
\begin{equation}
\bar{x}=\xi(x),
\qquad
\xi'(x)\neq0,
\end{equation}
together with the transformation
\begin{equation}
\textbf{z}=h(x)\textbf{U}(\bar{x}),
\end{equation}
where $h=(\xi')^{-1/2}$. Using $
\frac{d}{dx}=\xi'(x)\frac{d}{d\bar{x}},$ system (\ref{132}) reduces to 
\begin{equation}
\mathbf{U}_{\bar{x}\bar{x}}
+
(\xi')^{-2}
\left(
C-\frac12g(\bar{x})I
\right)\mathbf{U}=0,
\end{equation}
where
\begin{equation}
g
=
\frac{\xi'''}{\xi'}-\frac32
\left(
\frac{\xi''}{\xi'}
\right)^2
\end{equation}
is the Schwarzian derivative. Now, we decompose $C$ into trace and traceless parts as
\begin{equation}
C=
\widetilde{C}
+
\frac12(\operatorname{tr}C)I,
\qquad
\operatorname{tr}(\widetilde{C})=0,
\end{equation}
and then we choose $\xi$ such that
\begin{equation}
g
+
\operatorname{tr}(C)
=0.
\end{equation}

Hence, the scalar part disappears and the system becomes
\begin{equation}
\mathbf{U}''+K(x)\mathbf{U}=0,
\end{equation}
where $K(x)$ is given by
\begin{equation}
K
=
(\xi')^{-2}
\left(
C-\frac12(\operatorname{tr}C)I
\right),
\qquad
\operatorname{tr}(K)=0.
\end{equation}
This completes the proof.
\end{proof}

Herein, we consider the following Laguerre-Forsyth canonical  form for linear system of two second-order ODEs
\begin{equation}\label{8}
u_1^{\prime\prime}=m_{11}(x)u_1+m_{12}(x)u_2,~~~u_2^{\prime\prime}=m_{21}(x)u_1-m_{11}(x)u_2.
\end{equation}
 where the associated generalized  Wilczynski invariant matrix $W$ is  
\begin{equation}\label{170}
W=\begin{pmatrix}
m_{11}(x)&m_{12}(x)\\
m_{21}(x)&-m_{11}(x)
\end{pmatrix}.
\end{equation}
As we assumed $W$ has all entries different from zero, the coefficient functions $m_{11},m_{12},m_{21}$ are assumed to be nonvanishing on their domain and the corresponding system (\ref{8}) consists of fully coupled ODEs through the dependent variables $u_1, u_2$.\\

Consequently, we study the equivalence to the canonical form (\ref{8}) where the coefficient functions $m_{11},m_{12},m_{21}$ do not vanish. We start with the non-generic case where rank $W=1$, equivalently $m_{11}^2+m_{12}m_{21}=0$.

\section{Application of Cartan equivalence method}
	Let $(x,u_{1},u_{2},p_{1},p_{2})\in \mathbb{R}^{5} $, where $p_{1}=u^{\prime}_1$ and $p_{2}=u^{\prime}_2$, be the standard coordinates on the first Jet space $\textbf{J}^{1}(\mathbb{R},\mathbb{R}^{2})$. 
Consider the system of two second-order ODEs
\begin{equation}\label{NL}
u^{\prime\prime}_{1}=f^1(x,u_1,u_2,p_1,p_2),~~u^{\prime\prime}_{2}=f^2(x,u_1,u_2,p_1,p_2),~~
\end{equation}
where $f^1,f^2$ are smooth functions defined on an open subset $M$ of $\textbf{J}^{1}(\mathbb{R},\mathbb{R}^{2})$. Then the associated adapted coframe $\omega$ is given by
\begin{equation}\label{21}
\begin{pmatrix}
\omega_1\\
\omega_2\\
\omega_3\\
\omega_4\\
\omega_5
\end{pmatrix}=\Omega\begin{pmatrix}
du_1-p_1dx\\
du_2-p_2dx\\
dp_1-f^1dx\\
dp_2-f^2dx\\
dx
\end{pmatrix},
\end{equation}
where the invertible matrix $\Omega$ is defined as
\begin{equation}\label{87}
\Omega=
\begin{pmatrix}
1&0&0&0&0\\
0&1&0&0&0\\
I_1&I_2&1&0&0\\
I_3&I_4&0&1&0\\
0&0&0&0&1
\end{pmatrix},
\end{equation}
and 
\begin{equation}\label{106}
I_1=-\frac{1}{2}f^1_{p_1},~~ I_2=-\frac{1}{2}f^1_{p_2}, ~~I_3=-\frac{1}{2}f^2_{p_1},~~ I_4=-\frac{1}{2}f^2_{p_2}.
\end{equation}
The adapted coframe (\ref{21}) was used by Chern in \cite{R3} and by Fels in \cite{p15}. The equivalence problem of system (\ref{NL}) and the system 
\begin{equation}\label{S2}
\bar{u}^{\prime\prime}_{1}=\bar{f}^1(\bar{x},\bar{u}_1,\bar{u}_2,\bar{p}_1,\bar{p}_2),~~\bar{u}^{\prime\prime}_{2}=\bar{f}^2(\bar{x},\bar{u}_1,\bar{u}_2,\bar{p}_1,\bar{p}_2),
\end{equation}
  under an invertible point transformation
\begin{equation}\label{T1}
\Phi:~\bar{x}=\xi(x,u_1,u_2),~~\bar{u}_1=\phi_1(x,u_1,u_2),~~\bar{u}_2=\phi_2(x,u_1,u_2),
\end{equation}
can be reformulated in terms of the associated adapted coframes $\omega$ and $\bar{\omega}$ as 
\begin{equation}
\Phi^*
\begin{pmatrix}
\bar{\omega}_1\\
\bar{\omega}_2\\
\bar{\omega}_3\\
\bar{\omega}_4\\
\bar{\omega}_5
\end{pmatrix}=\begin{pmatrix}
a_1&a_2&0&0&0\\
a_3&a_4&0&0&0\\
a_5&a_6&a_7&a_8&0\\
a_9&a_{10}&a_{11}&a_{12}&0\\
a_{13}&a_{14}&0&0&a_{15}\\
\end{pmatrix}\begin{pmatrix}
\omega_1\\
\omega_2\\
\omega_3\\
\omega_4\\
\omega_5
\end{pmatrix},
\end{equation}
where $\Phi^*$ is the pullback of the first prolongation of the equivalence map $\Phi$. The auxiliary functions $a_i=a_i(x,u_1,u_2,p_1,p_2),i\in\{1,2,3,4,15\}$ can be evaluated, explicitly as 
\begin{equation}\label{as}
\begin{aligned}
&a_1=\phi_{1u_1}-\xi_{u_1}\frac{D_x\phi_1}{D_x\xi},~~a_2=\phi_{1u_2}-\xi_{u_2}\frac{D_x\phi_1}{D_x\xi},~~a_3=\phi_{2u_1}-\xi_{u_1}\frac{D_x\phi_2}{D_x\xi},\\
&a_4=\phi_{2u_2}-\xi_{u_2}\frac{D_x\phi_2}{D_x\xi}, ~~a_{15}=D_x\xi.
\end{aligned}\end{equation}
The associated structure group $G$ is defined as
\begin{equation}\label{G}
 G=\left\{\left.\begin{pmatrix}
a_1&a_2&0&0&0\\
a_3&a_4&0&0&0\\
a_5&a_6&a_7&a_8&0\\
a_9&a_{10}&a_{11}&a_{12}&0\\
a_{13}&a_{14}&0&0&a_{15}\\
\end{pmatrix}\right|~~ a_{15}
\begin{vmatrix}
a_1&a_2\\
a_3&a_4
\end{vmatrix}
\begin{vmatrix}
a_7&a_8\\
a_{11}&a_{12}
\end{vmatrix}\neq0\right\}.
\end{equation}\\
One can introduce the lifted coframe $\theta=(\theta_1, \theta_2, \theta_3,\theta_4, \theta_5)$ as 
\begin{equation}\label{S3}
\theta=S\omega,
\end{equation}
where $S:M\rightarrow G$ is a smooth function on $M$ taking values in the Lie subgroup $G$ of $GL(5,\mathbb{R})$.\\

After absorption, we have the following structure equations
\begin{equation}
d\begin{pmatrix}
\theta_1\\
\theta_2\\
\theta_3\\
\theta_4\\
\theta_5
\end{pmatrix}
=\begin{pmatrix}
\alpha_1&\alpha_2&0&0&0\\
\alpha_3&\alpha_4&0&0&0\\
\alpha_5&\alpha_6&\alpha_7&\alpha_8&0\\
\alpha_9&\alpha_{10}&\alpha_{11}&\alpha_{12}&0\\
\alpha_{13}&\alpha_{14}&0&0&\alpha_{15}\\
\end{pmatrix}\wedge
\begin{pmatrix}
\theta_1\\
\theta_2\\
\theta_3\\
\theta_4\\
\theta_5
\end{pmatrix}+\begin{pmatrix}
T^1_{35}\theta_3\wedge\theta_5+T^1_{45}\theta_4\wedge\theta_5\\
T^2_{35}\theta_3\wedge\theta_5+T^2_{45}\theta_4\wedge\theta_5\\
0\\
0\\
0\\
\end{pmatrix},
\end{equation}
where
\begin{equation}
\begin{aligned}
&T^1_{35}=\frac{a_1a_{12}-a_2a_{11}}{(a_8a_{11}-a_7a_{12})a_{15}},~T^1_{45}=\frac{a_2a_{7}-a_1a_{8}}{(a_8a_{11}-a_7a_{12})a_{15}},\\
&T^2_{35}=\frac{a_3a_{12}-a_4a_{11}}{(a_8a_{11}-a_7a_{12})a_{15}},~T^2_{45}=\frac{a_4a_{7}-a_3a_{8}}{(a_8a_{11}-a_7a_{12})a_{15}},\\
\end{aligned}
\end{equation}
 $\alpha_1,...,\alpha_{15}$ are the modified Maurer-Cartan forms and $\wedge$ is the wedge product operation.\\
 
We normalize the torsions $T^1_{35}=-1, T^1_{45}=0,  T^2_{35}=0, T^2_{45}=-1$ by setting
\begin{equation}
a_7=\frac{a_1}{a_{15}},~a_8=\frac{a_2}{a_{15}},~a_{11}=\frac{a_3}{a_{15}},~a_{12}=\frac{a_4}{a_{15}}.
\end{equation} 

This normalization reduces G to the 11-dimensional group $G_1$ by incorporating the values of $a_7,a_8,a_{11},a_{12}$ and yields an adapted coframe of the form (\ref{S3}), where $S:M\rightarrow G_1$.\\

The structure equations obtained after absorption in the second loop are
\begin{equation}
d\begin{pmatrix}
\theta_1\\
\theta_2\\
\theta_3\\
\theta_4\\
\theta_5
\end{pmatrix}
=\begin{pmatrix}
\alpha_1&\alpha_2&0&0&0\\
\alpha_3&\alpha_4&0&0&0\\
\alpha_5&\alpha_6&\alpha_1-\alpha_{11}&\alpha_2&0\\
\alpha_7&\alpha_{8}&\alpha_{3}&\alpha_{4}-\alpha_{11}&0\\
\alpha_{9}&\alpha_{10}&0&0&\alpha_{11}\\
\end{pmatrix}\wedge
\begin{pmatrix}
\theta_1\\
\theta_2\\
\theta_3\\
\theta_4\\
\theta_5
\end{pmatrix}+\begin{pmatrix}
-\theta_3\wedge\theta_5\\
-\theta_4\wedge\theta_5\\
T^3_{45}\theta_4\wedge\theta_5\\
T^4_{35}\theta_3\wedge\theta_5+T^4_{45}\theta_4\wedge\theta_5\\
0\\
\end{pmatrix},
\end{equation}
where 
\begin{equation}
T^3_{45}=\frac{2(a_2a_5-a_1a_6)}{a_1a_4-a_2a_3},~~T^4_{35}=\frac{2(a_3a_{10}-a_4a_9)}{a_1a_4-a_2a_3},~~T^4_{45}=-\frac{2(a_1a_{10}-a_2a_9+a_3a_6-a_4a_5)}{a_1a_4-a_2a_3}.
\end{equation}
Assume $a_1\neq0$ (see Remark (\ref{a1}) for the case $a_1= 0$). Vanishing the unbsorbable torsions $T^3_{45}, ~T^4_{35},~ T^4_{45}$  results in
\begin{equation}
a_6=\frac{a_2a_5}{a_1},~a_9=\frac{a_3a_5}{a_1},~a_{10}=\frac{a_4a_5}{a_1},
\end{equation}
and reduces $G_1$ to the 8-dimensional group $G_2$ by incorporating the values of $a_6, a_9, a_{10}$ in $G_1$.\\

In the third loop, the lifted coframe (\ref{S3}) where $S:M\rightarrow G_2$ has the structure equations, in the absorbed form,
\begin{equation}
d\begin{pmatrix}
\theta_1\\
\theta_2\\
\theta_3\\
\theta_4\\
\theta_5
\end{pmatrix}
=\begin{pmatrix}
\alpha_1&\alpha_2&0&0&0\\
\alpha_3&\alpha_4&0&0&0\\
\alpha_5&0&\alpha_1-\alpha_{8}&\alpha_2&0\\
0&\alpha_{5}&\alpha_{3}&\alpha_{4}-\alpha_{8}&0\\
\alpha_{6}&\alpha_{7}&0&0&\alpha_{8}\\
\end{pmatrix}\wedge
\begin{pmatrix}
\theta_1\\
\theta_2\\
\theta_3\\
\theta_4\\
\theta_5
\end{pmatrix}+\begin{pmatrix}
-\theta_3\wedge\theta_5\\
-\theta_4\wedge\theta_5\\
T^3_{25}\theta_2\wedge\theta_5\\
T^4_{15}\theta_1\wedge\theta_5+T^4_{25}\theta_2\wedge\theta_5\\
0\\
\end{pmatrix},
\end{equation}
where the invariants
\begin{equation}
\begin{aligned}
&T^3_{25}=\frac{J_1a^2_1+2J_2a_1a_2-J_3a^2_2}{(a_1a_4-a_2a_3)a^2_{15}},\\
& T^4_{15}=-\frac{J_1a^2_3+2J_2a_3a_4-J_3a^2_4}{(a_1a_4-a_2a_3)a^2_{15}}, ~T^4_{25}=\frac{2J_1a_1a_3+2J_2(a_1a_4+a_2a_3)-2J_3a_2a_4}{(a_1a_4-a_2a_3)a^2_{15}},\\
\end{aligned}
\end{equation}
are given in terms of
\begin{equation}\label{E40}
\begin{aligned}
&J_1=-f^1_{u_2}-I_2(I_1+I_4)-D_xI_2,\\
&J_2=\frac{1}{2}(f^1_{u_1}-f^2_{u_2}+I_1^2-I_4^2+D_x(I_1-I_4)),\\
&J_3=-f^2_{u_1}-I_3(I_1+I_4)-D_xI_3,\\
\end{aligned}
\end{equation}
 where $D_x=\frac{\partial}{\partial x}+p_1\frac{\partial}{\partial u_1}+p_2\frac{\partial}{\partial u_2}+f^1\frac{\partial}{\partial p_1}+f^2\frac{\partial}{\partial p_2}$. It is worth noting that $J_1,J_2,J_3$ constitute the entries of the generalized Wilczynski invariant matrix $W$ associated with system (\ref{NL}) as follows
\begin{equation}\label{W}
W=\begin{pmatrix}
J_2&-J_1\\
-J_3&-J_2\\
\end{pmatrix}.
\end{equation}
 Neut also presented, in Theorem \ref{Nuet3}, that vanishing of $J_1,J_2,J_3$ is a necessary condition for local equivalence to the free particle equations. 
 Moreover, 
\begin{equation}
\begin{pmatrix}
\frac{1}{2}T^4_{25}&-T^3_{25}\\
-T^4_{15}&-\frac{1}{2}T^4_{25}
\end{pmatrix}=\frac{1}{a_{15}^2}
\begin{pmatrix}
a_1&a_2\\
a_3&a_4
\end{pmatrix}\begin{pmatrix}
J_2&-J_1\\
-J_3&-J_2\\
\end{pmatrix}\begin{pmatrix}
a_1&a_2\\
a_3&a_4
\end{pmatrix}^{-1},
\end{equation}
and $T^3_{25}T^4_{15}+\frac{1}{4}(T^4_{25})^2=\frac{J_1J_3+J_2^2}{a^4_{15}}$.
This coincides with the fact that the determinant of $W$ is a relative invariant under point transformation, hence we introduce $J_4$ as
\begin{equation}\label{139}
J_4^4=J_1J_3+J_2^2=-\det W.
\end{equation}
As we exclude the trivial case $W=O$, we proceed through the following branch.\\

\textbf{Main branch: rank $W=1$}\\

In this branch the generalized Wilczynski invariant matrix $W$ is singular ($J_4=0$). By assumption, $W$ has nonzero entries, therefore the torsions $T^4_{15}=1,~T^4_{25}=0$ can be normalized by setting 
\begin{equation}
a_2=\frac{J_2}{J_3}a_1,~a_4=\frac{a_1a_{15}^2+J_2a_3}{J_3},
\end{equation}
that reduces $G_2$ to the six-dimensional group $G_3$ by incorporating the values of $a_2, a_4$ in $G_2$.\\

In the fourth loop, the lifted coframe (\ref{S3}) where $S:M\rightarrow G_3$ has the following structure equations, after absorption, 
\begin{equation}
\begin{aligned}
&d\begin{pmatrix}
\theta_1\\
\theta_2\\
\theta_3\\
\theta_4\\
\theta_5
\end{pmatrix}
=\begin{pmatrix}
\alpha_6-2\alpha_5&0&0&0&0\\
\alpha_1&\alpha_6&0&0&0\\
\alpha_2&0&\alpha_{6}-3\alpha_5&0&0\\
0&\alpha_{2}&\alpha_{1}&\alpha_{6}-\alpha_{5}&0\\
\alpha_{4}&\alpha_{3}&0&0&\alpha_{5}\\
\end{pmatrix}\wedge
\begin{pmatrix}
\theta_1\\
\theta_2\\
\theta_3\\
\theta_4\\
\theta_5
\end{pmatrix}\\
&~~~~~~~+\begin{pmatrix}
\theta_2\wedge(T^1_{23}\theta_3+T^1_{24}\theta_4+T^1_{25}\theta_5)-\theta_3\wedge\theta_5\\
-\theta_4\wedge\theta_5\\
T^3_{24}\theta_2\wedge\theta_4+T^3_{34}\theta_3\wedge\theta_4+T^3_{35}\theta_3\wedge\theta_5+T^1_{25}\theta_4\wedge\theta_5\\
\theta_1\wedge\theta_5+T^4_{24}\theta_2\wedge\theta_4+T^4_{34}\theta_3\wedge\theta_4+T^3_{35}\theta_4\wedge\theta_5\\
-T^4_{34}\theta_3\wedge\theta_5+(T^1_{23}+T^3_{34})\theta_4\wedge\theta_5\\
\end{pmatrix},
\end{aligned}
\end{equation}
where the invariants
\begin{equation}
\begin{aligned}
&T^1_{23}=\frac{J_2a_{13}-J_3a_{14}}{a_1a_{15}^2}-\frac{K_3}{a_1a_{15}}~~\text{mod}~T^1_{24},
~~~~~~T^1_{24}=\frac{K_2}{a_1a_{15}^3},~~~~~~T^1_{25}=\frac{K_1^3}{a_{15}^3},\\
&T^3_{24}=-\frac{(J_2a_{13}-J_3a_{14})K_1^3}{a_1a_{15}^5}+\frac{K_6}{a_1a_{15}^4}~~\text{mod}~T^1_{24},\\
&T^3_{34}=\frac{J_2a_{13}-J_3a_{14}}{2a_1a_{15}^2}+\frac{K_5}{2a_1a_{15}}~~\text{mod}~T^1_{24},~~~T^3_{35}=\frac{K_1^3a_3}{a_1a_{15}^3}+\frac{K_4}{a_{15}}-\frac{2a_5}{a_1},\\
&T^4_{24}=\frac{2(J_2a_{13}-J_3a_{14})(2a_5a_{15}-K_4a_1)+(3K_3+2K_5)a_5}{a_1^2a_{15}^3}+\frac{K_1^3a_{13}}{a_1a_{15}^3}+\frac{K_8}{a_1a_{15}^2}
~~\text{mod }~T^3_{24},\\
&T^4_{34}=\frac{2K_7a_{15}+a_{13}}{2a_1}~~\text{mod}~(T^1_{24},T^3_{34}),\\
\end{aligned}
\end{equation}
are given in term of $I_1,I_2,I_3,I_4,J_1,J_2,J_3$ given in (\ref{106}), (\ref{E40}) and their derivatives as follows
\begin{equation}\label{E41}
\begin{aligned}
&K_1^3=\frac{I_2J_3^2-(I_1-I_4)J_2J_3-I_3J_2^2}{J_3}-J_3D_x\bigg(\frac{J_2}{J_3}\bigg),~~~K_2=J_2J_3\bigg(\frac{J_2}{J_3}\bigg)_{p_1}-J_3^2\bigg(\frac{J_2}{J_3}\bigg)_{p_2},\\
&K_3=-J_3\bigg(\frac{J_2}{J_3}\bigg)_{p_1},~~~K_4=-\frac{(I_1-I_4)J_3+2I_3J_2-D_xJ_3}{2J_3},~~~K_5=\frac{J_3(2J_{2p_1}+J_{3p_2})-3J_2J_{3p_1}}{2J_3},~~~\\
&K_6=\frac{J_2^2J_3(I_{4p_1}-I_{1p_1})+J_2J_3^2(I_{1p_2}+I_{2p_1})-J_2^3I_{3p_1}-J_3^3I_{2p_2}}{J_3}+I_1J_2J_3\bigg(\frac{J_2}{J_3}\bigg)_{p_1}-I_2J_3^2\bigg(\frac{J_2}{J_3}\bigg)_{p_1}\\
&-J_3^2\bigg(\frac{I_4J_2}{J_3}\bigg)_{p_2}+J_2J_3\bigg(\frac{I_3J_2}{J_3}\bigg)_{p_2}-J_2J_3\bigg(\frac{J_2}{J_3}\bigg)_{u_1}+J_3^2\bigg(\frac{J_2}{J_3}\bigg)_{u_2},~~~K_7=-\frac{J_{3p_1}}{2J_3},~~~\\
&K_8=-\frac{J_3(I_1J_{2p_1}+I_3J_{2p_2}-J_3I_{1p_2}-2J_2I_{3p_2}-J_{2u_1}+J_{3u_2}-2J_2I_{3p_2})+J_2^2I_{3p_1}}{J_3}\\
&-J_3^2\bigg(\frac{I_2}{J_3}\bigg)_{p_1}-J_3^2\bigg(\frac{I_4}{J_3}\bigg)_{p_2}.\\
\end{aligned}
\end{equation}
We note that the relative invariant $K_2$ vanishes for the canonical form (\ref{8}).  Moreover, we have
\begin{equation}
\label{eq:array}
\left.
\begin{array}{l}
T^4_{24}-2T^1_{25}T^4_{34}+4T^3_{34}T^3_{35}=\frac{(2K_1^3K_5+K_6)a_3}{a_1^2a_{15}^4}+\frac{3(K_3-2K_5)a_5}{a_1^2a_{15}}-\frac{2K_1^3K_7-4K_4K_5-K_8}{a_1a_{15}^2}
\end{array}
\right.\text{mod}~T^1_{24}.
\end{equation} So $K_2$, $2K_1^3K_5+K_6$, $K_3-2K_5$ and $2K_1^3K_7-4K_4K_5-K_8$ form system of relative invariants and they are all identically zero for the canonical form (\ref{8}). Therefore, it belongs to the sub-branch 
\begin{equation}
K_2=0,~~K_3=2K_5,~~K_6=-2K_1^3K_5,~~K_8=2K_1^3K_7-4K_4K_5.
\end{equation}
  While at the canonical form (\ref{8}), the relative invariant 
\begin{equation}
K_1=\bigg(\frac{m_{11}m_{21}^{\prime}-m_{11}^{\prime}m_{21}}{m_{21}}\bigg)^{(1/3)}.
\end{equation}
 Thus, $K_1$ depends on the canonical form and we consider the following sub-branch.\\

\subsection{ Sub-branch $K_1=0$}

Normalizing the torsions $ T^3_{34}=0, T^3_{35}=0, T^4_{34}=0$ by setting 
\begin{equation}
a_5=\frac{K_4a_1}{2a_{15}}, ~a_{13}=-2K_7a_{15}, ~a_{14}=-\frac{2(J_2K_7-K_5)a_{15}}{J_3}.
\end{equation}
 reduces $G_3$ to the three-dimensional group $G_4$ by incorporating the values of $a_5, a_{13}, a_{14}$ in $G_3$ and yields the lifted coframe (\ref{S3}) where $S:M\rightarrow G_4$.\\

In the fifth loop, the structure equations, after absorption, are
\begin{equation}
\begin{aligned}
&d\begin{pmatrix}
\theta_1\\
\theta_2\\
\theta_3\\
\theta_4\\
\theta_5
\end{pmatrix}
=\begin{pmatrix}
\alpha_1-2\alpha_3&0&0&0&0\\
\alpha_2&\alpha_1&0&0&0\\
0&0&\alpha_1-3\alpha_3&0&0\\
0&0&\alpha_{2}&\alpha_1-\alpha_3&0\\
0&0&0&0&\alpha_3\\
\end{pmatrix}\wedge
\begin{pmatrix}
\theta_1\\
\theta_2\\
\theta_3\\
\theta_4\\
\theta_5
\end{pmatrix}\\
&~~~~~~~+\begin{pmatrix}
-\theta_3\wedge\theta_5\\
-\theta_4\wedge\theta_5\\
\theta_1\wedge(T^3_{12}\theta_2+T^3_{14}\theta_4+T^3_{15}\theta_5)\\
\theta_1\wedge(T^4_{12}\theta_2+\theta_5)+\theta_2\wedge(T^4_{13}\theta_3+T^3_{14}\theta_4+T^3_{15}\theta_5)\\
\sum^5_{i=2}T^5_{1i}\theta_1\wedge\theta_i+\sum^5_{i=3}T^5_{2i}\theta_2\wedge\theta_i\\
\end{pmatrix},
\end{aligned}
\end{equation}
where the invariants 
\begin{equation}
\begin{aligned}
&T^3_{12}=\frac{L_3}{a_1a_{15}^3},~~~
T^3_{14}=\frac{L_2}{a_1a_{15}^2},~~~T^3_{15}=\frac{L_1^2}{a_{15}^2},~~~T^4_{12}=\frac{L_3a_3+L_5a_1a_{15}^2}{a_1^2a_{15}^3},\\
&T^4_{23}=-\frac{3L_2a_3-2L_4a_1a_{15}^2}{2a_1^2a_{15}^2},~~~T^5_{12}=\frac{L_{12}}{a_1^2a_{15}},~~~\\
&T^5_{13}=\frac{L_7a_3^2-(L_8+L_{10})a_1a_3a_{15}^2+L_{11}a_1^2a_{15}^4}{a_1^4a_{15}^2},~~~T^5_{14}=-\frac{L_7a_3-L_{10}a_1a_{15}^2}{a_1^3a_{15}^2},~~~\\
&T^5_{15}=-\frac{(L_2+2L_6)a_3-2L_9a_1a_{15}^2}{2a_1^2a_{15}^2},~~~T^5_{23}=-\frac{L_7a_3-L_8a_1a_{15}^2}{a_1^3a_{15}^2},\\
&T^5_{24}=\frac{L_{7}}{a_1^2a_{15}^2},~~~T^5_{25}=\frac{L_{6}}{a_1a_{15}^2},
\end{aligned}
\end{equation}
are given in terms of 
\begin{equation}\label{120}
L_1^2=\frac{1}{4}K_4^2-\bigg(I_2I_3+I_4^2+J_2+f^2_{u_2}+D_x(I_4+\frac{1}{2}K_4)\bigg),\\
\end{equation}
and $L_2,\dots,L_{12}$ are given in Appendix B.\\

It should be noted that $ L_2,L_3L_4, L_5,L_6,L_7, L_8,L_9, L_{10}, L_{11},L_{12}$ form a system of relative invariants that all vanish at the canonical form (\ref{8}) which belongs to this sub-branch. Thus, we go through the sub-branch 
\begin{equation} 
 L_2=L_3= L_4= L_5=L_6=L_7=L_8= L_9= L_{10}= L_{11}=L_{12}=0.
\end{equation}
Moreover, the relative invariant $L_1$ at the canonical form (\ref{8}) is
\begin{equation}
L_1=\frac{\sqrt{5(m'_{21})^2-4m_{21}m^{\prime\prime}_{21}}}{4m_{21}}
\end{equation}
and we consider the following sub-branch.

\subsubsection{Sub-branch $L_1=0$}
In this sub-branch, the relative invariant $L_1$ vanishes. So in this loop of reduction, the lifted coframe
\begin{equation}
\begin{pmatrix}
\theta_1\\
\theta_2\\
\theta_3\\
\theta_4\\
\theta_5
\end{pmatrix}=\begin{pmatrix}
a_1&\frac{J_2a_1}{J_3}&0&0&0\\
a_3&\frac{J_2a_3+a_1a_{15}^2}{J_3}&0&0&0\\
\frac{K_4a_1}{2a_{15}}&\frac{J_2K_4a_1}{2J_3a_{15}}&\frac{a_1}{a_{15}}&\frac{J_2a_1}{J_3a_{15}}&0\\
\frac{K_4a_3}{2a_{15}}&\frac{(J_2a_3+a_1a_{15}^2)K_4}{2J_3a_{15}}&\frac{a_3}{a_{15}}&\frac{J_2a_3+a_1a_{15}^2}{J_3a_{15}}&0\\
-2K_7a_{15}&\frac{2(K_5-J_2K_7)a_{15}}{J_3}&0&0&a_{15}
\end{pmatrix}\begin{pmatrix}
\omega_1\\
\omega_2\\
\omega_3\\
\omega_4\\
\omega_5
\end{pmatrix}.
\end{equation}
has constant torsions and hence no group dependent invariants. So,  we cannot normalize the remaining group parameters $a_1, a_3,a_{15}$. \\

Moreover, the modified Maurer-Cartan forms $\alpha_1, \alpha_2,\alpha_3$ are uniquely determined and the problem is determinant. The e-structure on the 
eight-dimensional prolonged space $M^{(1)}=M\times G_5$ is $(\theta,\alpha)$ where 
\begin{equation}\label{135}
\begin{pmatrix}
\theta\\
\alpha\\
\end{pmatrix}=
\begin{pmatrix}
S&0\\
\Gamma&\Lambda\\
\end{pmatrix}
\begin{pmatrix}
\omega\\
V
\end{pmatrix},
\end{equation}
$\alpha=(\alpha_1,\alpha_2,\alpha_3)^T, V=(da_1, da_3,da_{15})^T,$ and
\begin{equation}\label{125}
\begin{aligned}
&S=\begin{pmatrix}
a_1&\frac{J_2a_1}{J_3}&0&0&0\\
a_3&\frac{J_2a_3+a_1a_{15}^2}{J_3}&0&0&0\\
\frac{K_4a_1}{2a_{15}}&\frac{J_2K_4a_1}{2J_3a_{15}}&\frac{a_1}{a_{15}}&\frac{J_2a_1}{J_3a_{15}}&0\\
\frac{K_4a_3}{2a_{15}}&\frac{(J_2a_3+a_1a_{15}^2)K_4}{2J_3a_{15}}&\frac{a_3}{a_{15}}&\frac{J_2a_3+a_1a_{15}^2}{J_3a_{15}}&0\\
-2K_7a_{15}&\frac{2(K_5-J_2K_7)a_{15}}{J_3}&0&0&a_{15}
\end{pmatrix},\\
&\Lambda=\begin{pmatrix}
\frac{1}{a_1}&0&\frac{2}{a_{15}}\\
-\frac{a_3}{a_1^2}&\frac{1}{a_1}&-\frac{2a_3}{a_1a_{15}}\\
0&0&\frac{1}{a_{15}}\\
\end{pmatrix},\\
 &\Gamma=\begin{pmatrix}
-\frac{J_2I_{3p_1}+J_3R_1}{J_3}&-\frac{J_2R_3+J_3R_5+2K_4K_5}{J_3}&2K_7&\frac{2(J_2K_7-K_5)}{J_3}&-\frac{2I_3J_2+J_3R_7}{2J_3}\\
\frac{2J_3(R_3-I_{3p_2})a_3-I_{3p_1}a_1a_{15}^2}{J_3a_1}&\frac{(2J_3R_4-2R_6)a_3-R_2a_1a_{15}^2}{J_3a_1}&-\frac{4K_7a_3}{a_1}&-\frac{2K_7a_1a_{15}^2+4R_6a_3}{J_3a_1}&\frac{2J_3K_4a_3-I_3a_1a_{15}^2}{J_3a_1}\\
I_{3p_2}+R_{2}&\frac{-J_3R_4+K_4R_6}{J_3}&2K_7&\frac{2R_6}{J_3}&-K_4
\end{pmatrix},
\end{aligned}
\end{equation}
  where $I_3,J_2, J_3, K_4, K_5,K_6,K_7$ are given in (\ref{106}), (\ref{E40}), (\ref{E41})  and 
\begin{equation}\label{190}
\begin{aligned}
&R_1=(I_1+3I_4+2K_4)_{p_1}-3I_{3p_2}-K_4K_7,~~~R_2=(I_4+\frac{1}{2}K_4)_{p_1}+K_4K_7,~~~\\
&R_3=(I_4+\frac{1}{2}K_4)_{p_1}-K_4K_7, ~~~
R_4=(I_1+\frac{1}{2}K_4)_{p_2}-I_{2p_1},~~~R_5=(3I_1+\frac{3}{2}K_4)_{p_2}-2I_{2p_1},~~~\\
&R_6=J_2K_7-K_5,~~~R_7=2I_1+5K_4.
\end{aligned}
\end{equation}

Any canonical form (\ref{8}) that belongs to this sub-branch has the following constant structure equations 
\begin{equation}
\begin{aligned}
&d\theta_1=-\theta_1\wedge(\alpha_1-2\alpha_3)-\theta_3\wedge\theta_5,\\
&d\theta_2=-\theta_1\wedge\alpha_2-\theta_2\wedge\alpha_1-\theta_4\wedge\theta_5,\\
&d\theta_3=-\theta_3\wedge(\alpha_1-3\alpha_3),\\
&d\theta_4=\theta_1\wedge\theta_5-\theta_3\wedge\alpha_2-\theta_4\wedge(\alpha_1-\alpha_3),\\
&d\theta_5=-\theta_5\wedge\alpha_3,\\
&d\alpha_1=0,\\
&d\alpha_2=-2\alpha_2\wedge\alpha_3,\\
&d\alpha_3=0.
\end{aligned}
\end{equation} 
 Therefore, the canonical form (\ref{8}) belongs to this sub-branch has the $\text{zero}^{th}$-order invariant coframe (\ref{135}) with rank zero. Hence, using \cite[Theorem 8.19, page 271, Theorem 8.22 page 275]{R6}, it admits a symmetry group of dimension $8$.  This proves Theorem \ref{thm1}.

\begin{remark}\label{Rm1} 
 The transformation 
\begin{equation}\label{tr1}
\bar{x}=x, ~\bar{u}_1=u_1+u_2,~\bar{u}_2=u_1-u_2,
\end{equation}
  maps $W$ given in (\ref{W}) to a similar matrix $\overline{\vphantom{x}W}$ as
\begin{equation}
\overline{\vphantom{x}W}=\begin{pmatrix}
1&1\\
1&-1\\
\end{pmatrix}\begin{pmatrix}
J_2&-J_1\\
-J_3&-J_2\\
\end{pmatrix}\begin{pmatrix}
1&1\\
1&-1
\end{pmatrix}^{-1}=-\frac{1}{2}\begin{pmatrix}
J_1+J_3&-J_1-2J_2+J_3\\
J_1-2J_2-J_3&-J_1-J_3
\end{pmatrix}.
\end{equation}
Moreover, any system of two second-order ODEs whose generalized Wilczynski invariant matrix $W$ has rank one and contains zero entries satisfies $J_1\neq0,J_2=J_3=0$, or $J_1=J_2=0,J_3\neq0$. In both cases, (\ref{tr1}) transforms the original system into an equivalent system of two second-order ODEs whose generalized Wilczynski invariant matrix has rank one and all entries different from zero. This agrees with Remark (\ref{rm10}).
\end{remark}
\begin{remark}\label{a1}
Since $\det J=\begin{vmatrix}
a_1&a_2\\
a_3&a_4
\end{vmatrix}\neq0$, if $a_1=\phi_{1u_1}-\xi_{u_1}\frac{D_x\phi_1}{D_x\xi}=0$, then $a_3=\phi_{2u_1}-\xi_{u_1}\frac{D_x\phi_2}{D_x\xi}\neq0$. So, if the system is linearizable by a transformation with $a_1=0$, then it will be linearizable by another transformation with $a_1\neq0$ which is the composition of the first one and 
\begin{equation}\label{tr1-1}
\bar{x}=x, ~\bar{u}_1=u_2,~\bar{u}_2=u_1.
\end{equation}
\end{remark}

\section{Main Results}
In this section, we utilise the previous results to derive the canonical form for linearizable system of two second-order ODEs that belongs to \text{Branch I}, characterized by $\text{rank}~W=1, K_1=0,L_1=0$. Then we determine the necessary and sufficient conditions for equivalence to this canonical form. 
\subsection{Canonical Forms }

\begin{theorem}
Any linear system whose generalized Wilczynski invariant matrix $W$ has rank one and  that satisfies the conditions $K_1=0,L_1=0$ is equivalent to
\begin{equation}\label{cf}
u_1^{\prime\prime}=u_1+u_2,~~u_2^{\prime\prime}=-(u_1+u_2),
\end{equation}
 where  $W,K_1,L_1$ are given in (\ref{W}),(\ref{E41}),(\ref{120}).\end{theorem}
\begin{proof}
Consider the Laguerre-Forsyth canonical form (\ref{8}) with $m_{11},m_{12},m_{21}$ all different from zero. By direct computations, one can show that the associated  generalized Wilczynski invariant matrix $W$ (\ref{170}) has rank one and the relative invariants $K_1,L_1$ vanish if and only if system (\ref{8}) has the form
\begin{equation}\label{177}
u_1^{\prime\prime}=\frac{c_1}{(c_2x+c_3)^4}(u_1+c_1u_2),~~u_2^{\prime\prime}=\frac{1}{(c_2x+c_3)^4}(u_1-c_1u_2), ~~c_2^2+c_3^2\neq0.
\end{equation} \\
Moreover, (\ref{177}) can be reduced further to 
\begin{equation}\label{178}
u_1^{\prime\prime}=u_1+u_2,~~~u_2^{\prime\prime}=-(u_1+u_2)~~~
\end{equation}
using different invertible transformations as explained in the following cases:
\begin{enumerate}
\item[(i)]  $c_1c_2\neq0$:
$$\bar{x}=-\frac{\sqrt{2c_1}}{2c_2(c_2x+c_3)},~\bar{u}_1=\frac{\sqrt{2c_1}}{2(c_2x+c_3)}\bigg(\frac{3c_2}{c_1}u_1-2c_2u_2\bigg),~\bar{u}_2=-\frac{c_2\sqrt{2}}{2(c_2x+c_3)\sqrt{c_1}}u_1,$$
 $\bigg|\frac{\partial (\bar{x},\bar{\textbf{u}})}{\partial (x,\textbf{u})}\bigg|=-\frac{c_2^2\sqrt{2c_1}}{2(c_2x+c_3)^4}\neq0.$\\
\item[(ii)] $c_1\neq0,~c_2=0$:
$$\bar{x}=\frac{\sqrt{c_1}}{c_3^2}x,~\bar{u}_1=\frac{1}{c_1}u_1,~\bar{u}_2=-u_2, ~~~\bigg|\frac{\partial (\bar{x},\bar{\textbf{u}})}{\partial (x,\textbf{u})}\bigg|=-\frac{1}{c_3^2\sqrt{c_1}}\neq0.$$ 
\item[(iii)] $c_1=0,~c_2\neq0$:
$$\bar{x}=-\frac{\sqrt{3}}{c_2x+c_3},~\bar{u}_1=-\frac{1}{c_2x+c_3}(u_1+3c_2^2u_2),~\bar{u}_2=\frac{3c_2^2}{c_2x+c_3}u_2,$$
~~~~$\bigg|\frac{\partial (\bar{x},\bar{\textbf{u}})}{\partial (x,\textbf{u})}\bigg|=-\frac{3c_2^3\sqrt{3}}{(c_2x+c_3)^4}\neq0.$
\item[(iv)] $c_1=0, c_2=0$:
$$\bar{x}=\frac{\sqrt{2}}{2c_3^2}x,~\bar{u}_1=u_1+u_2,~\bar{u}_2=u_1-u_2,~~~\bigg|\frac{\partial (\bar{x},\bar{\textbf{u}})}{\partial (x,\textbf{u})}\bigg|=-\frac{\sqrt{2}}{c_3^2}\neq0.$$
\end{enumerate}
It should be remarked that system (\ref{cf}) has only constant structure invariants and admits 8 point symmetries. 

\end{proof}
\begin{remark}
The canonical form (\ref{cf}) admits the following point symmetries
\begin{equation}
\begin{aligned}
&X_1=\partial_x, ~~X_2=\partial_{u_1}-\partial_{u_2}, ~~X_3=x\partial_{u_1}-x\partial_{u_2},~~X_4=u_1\partial_{u_1}+u_2\partial_{u_2},\\
& X_5=x\partial_{x}-2(u_1+u_2)\partial_{u_2},~~ X_6=u_2\partial_{u_1}-(2u_2+u_1)\partial_{u_2},~~\\
& X_7=\frac{1}{6}x^3\partial_{u_1}-(\frac{1}{6}x^3-x)\partial_{u_2},~~X_8=\frac{1}{2}x^2\partial_{u_1}+(1-\frac{1}{2}x^2)\partial_{u_2}.
\end{aligned}
\end{equation}
\end{remark}

\begin{remark}
The general solution of (\ref{cf}) is 
\begin{equation}\label{82}
u_1=\frac{1}{6}c_1x^3+\frac{1}{2}c_2x^2+c_3x+c_4,~u_2=-\frac{1}{6}c_1x^3-\frac{1}{2}c_2x^2+(c_1-c_3)x+(c_2-c_4).
\end{equation}
 So, solving the equivalence problem to the canonical form (\ref{cf}), then substitution back in (\ref{82}) will provide the solution of the linearizable one.\\
\end{remark}

\subsection{Main Theorem}
The main result of this paper is stated in the following theorem.
\begin{theorem}\label{thm1}
Any system of two second-order ODEs 
\begin{equation}\label{191}
u_1^{\prime\prime}=f^1(x,u_1,u_2,p_1,p_2),~~u_2^{\prime\prime}=f^2(x,u_1,u_2,p_1,p_2),
\end{equation}
whose generalize Wilczynski invariant matrix $W$ has nonzero entries is equivalent to the canonical form 
\begin{equation}\label{s2-1}
u_1^{\prime\prime}=u_1+u_2,~~u_2^{\prime\prime}=-(u_1+u_2),
\end{equation}
with eight  symmetries via a point transformation 
\begin{equation}\label{}
\Phi:~\bar{x}=\xi(x,u_1,u_2),~~\bar{u}_1=\phi_1(x,u_1,u_2),~~\bar{u}_2=\phi_2(x,u_1,u_2),
\end{equation}
with $a_1=\phi_{1u_1}-\xi_{u_1}\frac{D_x\phi_1}{D_x\xi}\neq0$ if and only if  \begin{equation}
\text{rank}~ W=1,
\end{equation}
 the relative invariants 
\begin{equation}
K_1,~K_2,~K_3-2K_5,~K_6,~K_8+4K_4K_5
\end{equation}
 evaluated at $f^1,f^2$ are identically zero and the associated invariant coframe $(\theta,\alpha)=(\theta_1,\theta_2,\theta_3,\theta_4,\theta_5,\alpha_1,\alpha_2,\alpha_3)$ defined on this branch has the constant structure equations
\begin{equation}\label{8sym}
\begin{aligned}
&d\theta_1=-\theta_1\wedge(\alpha_1-2\alpha_3)-\theta_3\wedge\theta_5,\\
&d\theta_2=-\theta_1\wedge\alpha_2-\theta_2\wedge\alpha_1-\theta_4\wedge\theta_5,\\
&d\theta_3=-\theta_3\wedge(\alpha_1-3\alpha_3),\\
&d\theta_4=\theta_1\wedge\theta_5-\theta_3\wedge\alpha_2-\theta_4\wedge(\alpha_1-\alpha_3),\\
&d\theta_5=-\theta_5\wedge\alpha_3,\\
&d\alpha_1=0,\\
&d\alpha_2=-2\alpha_2\wedge\alpha_3,\\
&d\alpha_3=0,
\end{aligned}
\end{equation} 
 where 
\begin{equation}\label{coframe}
\begin{pmatrix}
\theta\\
\alpha\\
\end{pmatrix}=
\begin{pmatrix}
S&0\\
\Gamma&\Lambda\\
\end{pmatrix}
\begin{pmatrix}
\omega\\
V
\end{pmatrix},
\end{equation}
$\alpha=(\alpha_1,\alpha_2,\alpha_3)^T, V=(da_1, da_3,da_{15})^T$ and
\begin{equation}\label{s3-10}
\begin{aligned}
&S=\begin{pmatrix}
a_1&\frac{J_2a_1}{J_3}&0&0&0\\
a_3&\frac{J_2a_3+a_1a_{15}^2}{J_3}&0&0&0\\
\frac{K_4a_1}{2a_{15}}&\frac{J_2K_4a_1}{2J_3a_{15}}&\frac{a_1}{a_{15}}&\frac{J_2a_1}{J_3a_{15}}&0\\
\frac{K_4a_3}{2a_{15}}&\frac{(J_2a_3+a_1a_{15}^2)K_4}{2J_3a_{15}}&\frac{a_3}{a_{15}}&\frac{J_2a_3+a_1a_{15}^2}{J_3a_{15}}&0\\
-2K_7a_{15}&\frac{2(K_5-J_2K_7)a_{15}}{J_3}&0&0&a_{15}
\end{pmatrix},\\
&\Lambda=\begin{pmatrix}
\frac{1}{a_1}&0&\frac{2}{a_{15}}\\
-\frac{a_3}{a_1^2}&\frac{1}{a_1}&-\frac{2a_3}{a_1a_{15}}\\
0&0&\frac{1}{a_{15}}\\
\end{pmatrix},\\
 &\Gamma=\begin{pmatrix}
-\frac{J_2I_{3p_1}+J_3R_1}{J_3}&-\frac{J_2R_3+J_3R_5+2K_4K_5}{J_3}&2K_7&\frac{2(J_2K_7-K_5)}{J_3}&-\frac{2I_3J_2+J_3R_7}{2J_3}\\
\frac{2J_3(R_3-I_{3p_2})a_3-I_{3p_1}a_1a_{15}^2}{J_3a_1}&\frac{(2J_3R_4-2R_6)a_3-R_2a_1a_{15}^2}{J_3a_1}&-\frac{4K_7a_3}{a_1}&-\frac{2K_7a_1a_{15}^2+4R_6a_3}{J_3a_1}&\frac{2J_3K_4a_3-I_3a_1a_{15}^2}{J_3a_1}\\
I_{3p_2}+R_{2}&\frac{-J_3R_4+K_4R_6}{J_3}&2K_7&\frac{2R_6}{J_3}&-K_4
\end{pmatrix},
\end{aligned}
\end{equation}
and $W,I_3,J_2, J_3, K_i,R_j, i=1,\dots,8, j=1,\dots,7$ are given in (\ref{W}),(\ref{106}), (\ref{E40}), (\ref{E41}) ,(\ref{190}).
\end{theorem}
\begin{remark}
The conditions $L_1=\dots=L_{12}=0$ are not stated explicitly in Theorem \ref{thm1}, as they are satisfied implicitly through the structure equations (\ref{8sym}).
\end{remark}


\section{Construction of the Linearizing Map Using the Invariant Prolongrd Coframe}
The linearizing point transformation of a linearizable system of two second-order ODEs whose generalize Wilczynski
invariant matrix $W$ has rank one and satisfies $K_1=0,L_1=0$ can be constructed using the invariant prolonged coframe. \\
\begin{proposition}\label{s3-8}
Assume that the system of two second-order ODEs
\begin{equation}\label{136}
u_1^{\prime\prime}=f^1(x,u_1,u_2,p_1,p_2),~~u_2^{\prime\prime}=f^2(x,u_1,u_2,p_1,p_2)
\end{equation}
 whose generalized Wilczynski invariant matrix $W$ has nonzero entries, is equivalent to the canonical form
\begin{equation}\label{138}
\bar{u}_1^{\prime\prime}=\bar{u}_1+\bar{u}_2,~~\bar{u}_2^{\prime\prime}=-(\bar{u}_1+\bar{u}_2)
\end{equation}
under the invertible point transformation
\begin{equation}\label{137}
\Phi:~~\bar{x}=\xi(x,u_j)~~\bar{u}_i=\phi_{i}(x,u_j).
\end{equation}
Then the linearizing point transformation (\ref{137}) where $\phi_{1u_1}-\xi_{u_1}\frac{D_x\phi_1}{D_x\xi}\neq0$
can be obtained by the following sequential procedure:\\
\textbf{Step 1.}  Determine the matrices 
\begin{equation}\label{s3-12}
B=(\bar{S})^{-1}S~\Omega, ~H=(\Lambda)^{-1}(\Gamma\Omega-\bar{\Gamma}B).
\end{equation}
 \textbf{Step 2.} Find the nonzero auxiliary functions $a_1,a_{15}$ and the auxiliary function $a_3$ by solving the system \\
\begin{equation}\label{168}
\begin{aligned}
&\begin{pmatrix}
a_{1u_1}&a_{1u_2}&a_{1p_1}&a_{1p_2}&D_xa_1\\
a_{3u_1}&a_{3u_2}&a_{3p_1}&a_{3p_2}&D_xa_3\\
a_{15u_1}&a_{15u_2}&a_{15p_1}&a_{15p_2}&D_xa_{15}\\
\end{pmatrix}=-H,
\end{aligned}
\end{equation}
\textbf{Step 3.} Find $\xi, \phi_i, \chi_i$ by solving the system
\begin{equation}\label{169}
\begin{aligned}
&\begin{pmatrix}
\phi_{iu_j}&\textbf{0}&D_x\phi_i\\
\chi_{iu_j}&\chi_{ip_j}&D_x\chi_i\\
\xi_{u_j}&\textbf{0}&D_x\xi
\end{pmatrix}=\begin{pmatrix}
I&\textbf{0}&\chi_i\\
\textbf{0}&I&\bar{f}^i\\
\textbf{0}&\textbf{0}&1\\
\end{pmatrix}B,
\end{aligned}
\end{equation}

where
\begin{equation}
\begin{aligned}
&\chi_i=\frac{D_x\phi_i}{D_x\xi},~~\bar{f}^1=\phi_1+\phi_2,~\bar{f}^2=-(\phi_1+\phi_2),\\
\end{aligned}
\end{equation}
 and $S,\Lambda,\Gamma$ are given in (\ref{s3-10}), $\Omega$ is given in (\ref{87}) and $\bar{S}, \bar{\Gamma}$ are the pullback of the matrices $S, \Gamma$ evaluated at the canonical form (\ref{138}), respectively, after substituting $\bar{a}_1=1, \bar{a}_3=0, \bar{a}_{15}=1$.
\end{proposition}
\begin{proof} 
Define the vectors $$U=(du_i-p_idx, ~dp_i-f^idx, ~dx)^{T},~~V=(da_1,da_3,da_{15})^{T}.$$ Then the symmetrical version of the coframe $(\theta,\alpha)$ given in (\ref{coframe}) provides
\begin{equation}
\begin{pmatrix}
\bar{S}&\textbf{0}\\
\bar{\Gamma}&\bar{\Lambda}\\
\end{pmatrix}\begin{pmatrix}
\bar{\Omega}&\textbf{0}\\
\textbf{0}&\text{I}\\
\end{pmatrix}\begin{pmatrix}
\bar{U}\\
\bar{V}\\
\end{pmatrix}=
\begin{pmatrix}
S&\textbf{0}\\
\Gamma&\Lambda\\
\end{pmatrix}\begin{pmatrix}
\Omega&\textbf{0}\\
\textbf{0}&\text{I}\\
\end{pmatrix}\begin{pmatrix}
U\\
V\\
\end{pmatrix},
\end{equation}
or equivalently,
\begin{equation}\label{141}
\begin{pmatrix}
\bar{U}\\
\bar{V}\\
\end{pmatrix}=
\begin{pmatrix}
(\bar{S}\bar{\Omega})^{-1}S\Omega&\textbf{0}\\
(\bar{\Lambda})^{-1}(\Gamma-\bar{\Gamma}(\bar{S})^{-1}S)\Omega&(\bar{\Lambda})^{-1}\Lambda\\
\end{pmatrix}\begin{pmatrix}
U\\
V\\
\end{pmatrix}.\\
\end{equation}
Substituting $\bar{a}_1=1, \bar{a}_3=0,\bar{a}_{15}=1$ in (\ref{141}) and then inserting the first prolongation of the point transformation (\ref{137}) with $\bar{p}_i=\chi_i=\frac{D_x\phi_i}{D_x\xi},$ results in
\begin{equation}\label{s3-11}
\begin{pmatrix}
\bar{U}\\
\textbf{0}\\
\end{pmatrix}=
\begin{pmatrix}
(\bar{S})^{-1}S\Omega&\textbf{0}\\
(\bar{\Lambda})^{-1}(\Gamma-\bar{\Gamma}(\bar{S})^{-1}S)\Omega&(\bar{\Lambda})^{-1}\Lambda\\
\end{pmatrix}\begin{pmatrix}
U\\
V\\
\end{pmatrix},
\end{equation}
where $\bar{U}=(d\phi_i-\chi_id\xi,~d\chi_i-\bar{f}^id\xi, ~d\xi)^{T}$ and
\begin{equation}
\bar{S}=\begin{pmatrix}
1&1&0&0&0\\
0&1&0&0&0\\
0&0&1&1&0\\
0&0&0&1&0\\
0&0&0&0&1
\end{pmatrix}, \bar{\Gamma}=\begin{pmatrix}
0&0&0&0&0\\
0&0&0&0&0\\
0&0&0&0&0\\
\end{pmatrix}, \bar{\Lambda}=\begin{pmatrix}
1&0&2\\
0&1&0\\
0&0&1
\end{pmatrix}.
\end{equation}
 Multiplying both sides of (\ref{s3-11}) by $\begin{pmatrix}
\mathbf{I}&0\\
0&\Lambda^{-1}\end{pmatrix}$  provides
\begin{equation}\label{142}
\begin{pmatrix}
\bar{U}\\
\textbf{0}\\
\end{pmatrix}=
\begin{pmatrix}
B&\textbf{0}\\
H&\textbf{I}\\
\end{pmatrix}\begin{pmatrix}
U\\
V\\
\end{pmatrix},
\end{equation}
where $B=(\bar{S})^{-1}S~\Omega, ~H=(\Lambda)^{-1}(\Gamma-\bar{\Gamma}(\bar{S})^{-1}S)\Omega$.\\

 Furthermore, using 
\begin{equation}
V=\begin{pmatrix}
a_{1u_1}&a_{1u_2}&a_{1p_1}&a_{1p_2}&D_xa_1\\
a_{3u_1}&a_{3u_2}&a_{3p_1}&a_{3p_2}&D_xa_3\\
a_{15u_1}&a_{15u_2}&a_{15p_1}&a_{15p_2}&D_xa_{15}\\
\end{pmatrix}U,
\end{equation}
together with (\ref{142}) provides (\ref{168}). 
Also, invoking
\begin{equation}
\overline{\vphantom{x}U}=\begin{pmatrix}
I&\textbf{0}&-\chi_i\\
\textbf{0}&I&-\bar{f}^i\\
\textbf{0}&\textbf{0}&1\\
\end{pmatrix}\begin{pmatrix}
d\phi_i\\
d\chi_i\\
d\xi
\end{pmatrix},
\end{equation}
one gets
\begin{equation}\label{167}
\overline{\vphantom{x}U}=\begin{pmatrix}
I&\textbf{0}&-\chi_i\\
\textbf{0}&I&-\bar{f}^i\\
\textbf{0}&\textbf{0}&1\\
\end{pmatrix}\begin{pmatrix}
\phi_{iu_j}&\textbf{0}&D_x\phi_i\\
\chi_{iu_j}&\chi_{ip_j}&D_x\chi_i\\
\xi_{u_j}&\textbf{0}&D_x\xi
\end{pmatrix}U.
\end{equation}
Hence, (\ref{142}) and (\ref{167}) provides (\ref{169}). 
This completes the proof.
\end{proof}
\vspace{0.5cm}

\begin{example}
 Consider the nonlinear system of second-order ODEs
\begin{equation}\label{s5}
u_1^{\prime\prime}={u^{'}_2}^{2}(u'_1+1)(x+u_1),~~u_2^{\prime\prime}={u^{'}_2}^3(x+u_1),
 \end{equation}
whose generalized Wilczynski invariant matrix is
\begin{equation}
W=\begin{pmatrix}
{u^{'}_2}^{2}(u'_1+1)&-u'_2(u'_1+1)^2\\
{u^{'}_2}^{3}&-{u^{'}_2}^{2}(u'_1+1)
\end{pmatrix}.
\end{equation}
The functions
\begin{equation}\label{s8}
 f^1=p^{2}_2(p_1+1)(x+u_1),~~f^2=p^{3}_2(x+u_1)
\end{equation}
 satisfy the conditions of Theorem \ref{thm1}. That implies system (\ref{s5}) admits eight Lie point symmetries and can be reduced to
\begin{equation}
\bar{u}_1^{\prime\prime}=\bar{u}_1+\bar{u}_2,~~\bar{u}_2^{\prime\prime}=-(\bar{u}_1+\bar{u}_2).
\end{equation}
The equivalence map can be constructed using Proposition \ref{s3-8} as follows:\\
 \textbf{Step 1.} We compute the matrices $B$ and $H$ given in (\ref{s3-12}). \\
\begin{equation}\label{E10}
B=\begin{pmatrix}
a_1-a_3&-\frac{((p_1+1)p_2^2-a_{15}^2)a_1-(p_1+1)p_2^2a_3}{p_2^3}&0&0&0\\
a_3&-\frac{a_1a_{15}^2+(p_1+1)a_3}{p_2^3}&0&0&0\\
0&-\frac{(x+u_1)a_1a_{15}}{p_2}&\frac{a_1-a_3}{a_{15}}&-\frac{((p_1+1)p_2^2-a_{15}^2)a_1-(p_1+1)p_2^2)a_3}{p_2^3a_{15}}&0\\
0&\frac{(x+u_1)a_1a_{15}}{p_2}&\frac{a_3}{a_{15}}&-\frac{a_1a_{15}^2+(p_1+1)a_3}{p_2^3a_{15}}&0\\
0&\frac{a_{15}}{p_2}&0&0&a_{15}\\
\end{pmatrix},
\end{equation}
and\begin{equation}\label{s3-13}
H=\begin{pmatrix}
0&0&0&0&0\\
0&-18(x+u_1)^2p_2^5a_3^2-3(x+u_1)p_2a_3&0&0&0\\
0&0&0&-\frac{a_{15}}{p_2}&-(x+u_1)p_2^2a_{15}\\
\end{pmatrix}.
\end{equation}
\textbf{Step 2.} Solving system (\ref{168}) provides the auxiliary functions  
\begin{equation}\label{102}
a_1=1,~ a_3=0,~a_{15}=p_2.
\end{equation}
\textbf{Step 3.} Substituting (\ref{102}) in  (\ref{E10}), then solving system (\ref{169}) results in  
$$\chi_1=\frac{p_1+p_2}{p_2}, ~\chi_2=\frac{1-p_2}{p_2}, ~ \phi_1=u_1+u_2, ~\phi_2=x-u_2,~\xi=u_2+1.$$
 Thus, the equivalence map is
\begin{equation}\label{}
\bar{x}=u_2+1, ~~~\bar{u}_1=u_1+u_2,~~~\bar{u}_2=x-u_2.
\end{equation}
\end{example}

\vspace{0.5cm}

\begin{example}
 Consider the nonlinear system of second-order ODEs
\begin{equation}\label{Os5}
u_1^{\prime\prime}=-\frac{(x^2u'_1-1)(xu_1+1){u^{'}_2}^{2}+2u'_1}{x},~~u_2^{\prime\prime}=-\frac{(2+(xu_1+1)x^2{u^{'}_2}^{2})u'_2}{x}
 \end{equation}
whose generalized Wilczynski invariant matrix is
\begin{equation}
W=\begin{pmatrix}
-{u^{'}_2}^{2}(x^2u'_1-1)&\frac{u'_2(x^2u'_1-1)^2}{x^2}\\
-x^2{u^{'}_2}^{3}&{u^{'}_2}^{2}(x^2u'_1-1)
\end{pmatrix}.
\end{equation}
The functions
\begin{equation}\label{Os8}
 f^1=-\frac{(x^2p_1-1)(xu_1+1)p^{2}_2+2p_1}{x},~~f^2=-\frac{(2+(xu_1+1)x^2p^{2}_2)p_2}{x}
\end{equation}
 satisfy the conditions of Theorem \ref{thm1}. That implies system (\ref{Os5}) admits eight Lie point symmetries and can be reduced to
\begin{equation}
\bar{u}_1^{\prime\prime}=\bar{u}_1+\bar{u}_2,~~\bar{u}_2^{\prime\prime}=-(\bar{u}_1+\bar{u}_2).
\end{equation}
The equivalence map can be constructed by using Proposition \ref{s3-8} as follows:
 \textbf{Step 1.} We compute the matrices $B$ and $H$ given in (\ref{s3-12}). \\
\begin{equation}\label{OE10}
B=\begin{pmatrix}
a_1-a_3&-\frac{\Delta+(x^2p_1^2-1)p_2^2}{x^2p_2^3a_1}&0&0&0\\
a_3&\frac{\Delta}{x^2p_2^3}&0&0&0\\
0&-\frac{(xu_1+1)a_1a_{15}}{p_2}&\frac{a_1-a_3}{a_{15}}&-\frac{\Delta+(x^2p_1-1)p_2^2a_1}{x^2p_2^3a_{15}}&0\\
0&\frac{(xu_1+1)a_1a_{15}}{xp_2}&\frac{a_3}{a_{15}}&-\frac{\Delta}{x^2p_2^3a_{15}}&0\\
0&\frac{a_{15}}{p_2}&0&0&a_{15}\\
\end{pmatrix},
\end{equation}
where $\Delta=a_1a_{15}^2-(x^2p_1-1)p_2^2a_3$ and\begin{equation}\label{Os3-13}
H=\begin{pmatrix}
0&0&0&0&0\\
0&\frac{(3x^2p_2^2(xu_1+1)+2)((6x^4u_1p_2^4+6x^3p_2^4+4xp_2^2)a_3+1)a_3}{xp_2}&0&0&0\\
0&0&0&-\frac{a_{15}}{p_2}&\frac{((xu_1+1)x^2p_2^2+2)a_{15}}{x}\\
\end{pmatrix}.
\end{equation}
\textbf{Step 2.} Solving system (\ref{168}) provides the auxiliary functions  
\begin{equation}\label{O102}
a_1=1,~ a_3=0,~a_{15}=p_2.
\end{equation}
\textbf{Step 3.} Substituting (\ref{O102}) in  (\ref{OE10}) and then solving system (\ref{169}) results in  
$$\chi_1=\frac{p_1}{p_2}, ~\chi_2=-\frac{1}{x^2p_2}, ~ \phi_1=u_1, ~\phi_2=\frac{1}{x},~\xi=u_2.$$
 Thus, the equivalence map is
\begin{equation}\label{O87}
\bar{x}=u_2, ~~~\bar{u}_1=u_1,~~~\bar{u}_2=\frac{1}{x}.
\end{equation}
\end{example}

\vspace{0.5cm}

\begin{example}
 Consider the nonlinear system of geodesic equations
\begin{equation}\label{b3}
u_1^{\prime\prime}=-x{u^{'}_2}^{2},~~u_2^{\prime\prime}=0.
 \end{equation}
whose generalized Wilczynski invariant matrix 
\begin{equation}
W=\begin{pmatrix}
0&-u'_2\\
0&0\end{pmatrix}\end{equation}
 has zero entries; hence applying the transformation (\ref{tr1}) reduces (\ref{b3}) to

\begin{equation}\label{OOs5}
u_1^{\prime\prime}=\frac{1}{4}x(u'_1-u'_2)^2,~~u_2^{\prime\prime}=\frac{1}{4}x(u'_1-u'_2)^2
 \end{equation}
whose generalized Wilczynski invariant matrix is
\begin{equation}
W=\begin{pmatrix}
-\frac{1}{2}(u'_1-u'_2)&\frac{1}{2}(u'_1-u'_2)\\
-\frac{1}{2}(u'_1-u'_2)&\frac{1}{2}(u'_1-u'_2)\end{pmatrix}.\end{equation}
The functions
\begin{equation}\label{Os8}
 f^1=\frac{1}{4}x(p_1-p_2)^2,~~f^2=\frac{1}{4}x(p_1-p_2)^2
\end{equation}
 satisfy the conditions of Theorem \ref{thm1}. That implies system (\ref{OOs5}) admits eight Lie point symmetries and can be reduced to
\begin{equation}\label{b4}
\bar{u}_1^{\prime\prime}=\bar{u}_1+\bar{u}_2,~~\bar{u}_2^{\prime\prime}=-(\bar{u}_1+\bar{u}_2).
\end{equation}
The equivalence map can be constructed using Proposition \ref{s3-8} as follows:
 \textbf{Step 1.} We compute the matrices $B$ and $H$ given in (\ref{s3-12}). \\
\begin{equation}\label{OOE10}
B=\begin{pmatrix}
a_1-a_3&-a_1+a_3-\frac{4a_1a_{15}^2}{p_1-p_2}&0&0&0\\
a_3&-a_3+\frac{4a_1a_{15}^2}{p_1-p_2}&0&0&0\\
xa_1a_{15}&-xa_1a_{15}&\frac{a_1-a_3}{a_{15}}&-\frac{a_1-a_3}{a_{15}}-\frac{4a_1a_{15}}{p_1-p_2}&0\\
-xa_1a_{15}&xa_1a_{15}&\frac{a_3}{a_{15}}&-\frac{a_3}{a_{15}}+\frac{4a_1a_{15}}{p_1-p_2}&0\\
\frac{a_{15}}{p_1-p_2}&\frac{a_{15}}{p_1-p_2}&0&0&a_{15}\\
\end{pmatrix},
\end{equation}
and\begin{equation}\label{OOs3-13}
H=\begin{pmatrix}
0&0&\frac{a_{1}}{p_1-p_2}&-\frac{a_{1}}{p_1-p_2}&0\\
0&\frac{(xa_1a_{15}^2+2)xa_1a^2_{15}}{2(p_1-p_2)}&\frac{a_3}{p_1-p_2}&\frac{4a_1a^2_{15}-(p_1-p_2)a_3}{(p_1-p_2)^2}&xa_1a_{15}^2\\
0&0&-\frac{a_{15}}{p_1-p_2}&-\frac{a_{15}}{p_1-p_2}&0\\
\end{pmatrix}.
\end{equation}
\textbf{Step 2.} Solving system (\ref{168}) provides the auxiliary functions  
\begin{equation}\label{OO102}
a_1=\frac{1}{p_1-p_2},~ a_3=\frac{1-4p_2}{p_1-p_2},~a_{15}=p_1-p_2.
\end{equation}
\textbf{Step 3.} Substituting (\ref{OO102}) in  (\ref{OOE10}), then solving system (\ref{169}) results in  
$$\chi_1=\frac{4p_1}{p_1-p_2}, ~\chi_2=-\frac{1-4p_1}{p_1-p_2}, ~ \phi_1=-4u_1, ~\phi_2=-x+4u_1,~\xi=u_1-u_2+1.$$
 Thus, the equivalence map that reduces (\ref{b3}) to the canonical form (\ref{b4}) is the composition of (\ref{tr1}) and 
\begin{equation}\label{O87}
\bar{x}=u_1-u_2+1, ~~~\bar{u}_1=-4u_1,~~~\bar{u}_2=-x+4u_1.
\end{equation}
\end{example}


\section{Conclusion}
The present investigation constitutes a step towards solving the lineariziation problem for a system of two second-order ODEs by means of Cartan's method under point transformation. Here we considered a linearizable class of systems admitting an eight-dimensional point symmetry group. We determined the canonical form and established the necessary and sufficient conditions for equivalence to this canonical form. We then presented a systematic method for constructing the equivalence map.
Examples were given to illustrate our approach.
\subsection*{Acknowledgements}
The authors would like to thank Birzeit University for its support and excellent research facilities. FM is grateful to Wits.\\\\
Conflict of Interest: The authors declare that they have no conflict of interest.

\appendix
\footnotesize

\section*{Appendix A.}
\begin{center}
\begin{tikzpicture}[
  level distance=1.5cm,
  level 1/.style={sibling distance=4cm, every node/.style={fill=cyan!20}},
  level 2/.style={sibling distance=5cm, every node/.style={fill=blue!20}},
  level 3/.style={sibling distance=3cm, every node/.style={fill=red!20}},
 level 4/.style={sibling distance=2cm, every node/.style={fill=green!20}},
level 5/.style={sibling distance=2cm, every node/.style={fill=purple!20}},
  every node/.style={draw, rectangle, rounded corners=5pt, align=center}
]

\node[align=center, fill=violet!20] {Linearizable system\\$u''_1=f^1(x,u_1,u_2,p_1,p_2),$\\$u''_2=f^2(x,u_1,u_2,p_1,p_2),$
}
 child {node[align=center, rounded corners=5pt]  {~~Rank $W=0$~~}
 child {node[align=center, rounded corners=5pt]  { $\tilde{S}=O$}
child {node[align=center,fill=purple!20] {admits fifteen\\point \\symmetries}}}}
  child {node[align=center, rounded corners=5pt]  {Rank $W=1$}
 child {node[align=center, rounded corners=5pt]   {$K_2=0,~K_3=2K_5$,\\
$K_6=-2K_1^3K_5$,\\
$K_8=2K_1^3K_7-4K_4K_5$}
   child {node[align=center, rounded corners=5pt]   {$~~~K_1=0~~~$}
  child {node[align=center, rounded corners=5pt,fill=green!20]   {$~~~L_1=0~~~$}
child {node[align=center] {has structure equations (\ref{8sym})\\ with constant invariants}
child {node[align=center] {admits eight\\point symmetries}}}
}
  }
}}
 child {node[align=center, rounded corners=5pt]  {~~Rank $W=2$~~}};
\end{tikzpicture}\\
\end{center}
 where
$\tilde{S},K_1,\dots,K_{8}, L_1$ are given in (\ref{Fels}), (\ref{E41}), (\ref{120}).
\section*{Appendix B.}
\begin{equation}
\begin{aligned}
&L_2=\frac{1}{2J_2}\bigg(J_2^2(2I_4+K_4)_{p_1}-J_2J_3(2I_4+K_4)_{p_2}+2J_2K_5(2I_4+K_4)+2(J_3J_{2u_2}-J_2J_{3u_2})\\&+2J_3(J_2I_{2p_1}-J_3I_{2p_2}+2I_2K_5)\bigg),\\
\end{aligned}\end{equation}
\begin{equation}
\begin{aligned}
&L_3=\frac{1}{4}J_3(2I_4+K_4)(2I_1+K_4)_{p_2}-\frac{1}{4}J_2(2I_1+K_4)(2I_4+K_4)_{p_1}+\frac{1}{2}J_2I_{3p_2}(2I_4+K_4)\\
&-\frac{1}{2}J_3I_{2p_1}(2I_1+K_4)+\frac{1}{2}(I_2J_3(2I_1+K_4)_{p_1}-I_3J_2(2I_4+K_4)_{p_2})+\frac{1}{2}(J_2K_{4u_1}-J_3K_{4u_2})\\
&+2K_5(I_4^2-\frac{1}{4}K_4^2+J_2+f^2_{u_2}+I_2I_3+\frac{1}{2}D_x(2I_4+K_4))+J_3(I_{2u_1}-I_{1u_2})+J_2(I_{4u_1}-I_{3u_2})\\
&+(I_2J_2I_{3p_1}-I_3J_3I_{2p_2}),\\
&L_4=-\frac{1}{4J_2^2}\bigg(J_2^2(3(2I_4+K_4)_{p_1}-2K_7(2I_4+K_4))+8J_2K_5(2I_4+K_4)\\
&+2J_3(2I_2K_5-J_3I_{2p_2})+4J_2J_3((I_1-I_4)_{p_2}-I_2K_7)+2(J_{2u_2}J_3-2J_2J_{3u_2})\bigg),\\
&L_5=-\frac{1}{4}(2I_1+K_4)(2I_4+K_4)_{p_1}+\frac{1}{2}(I_{3p_2}(2I_4+K_4)-I_3(2I_4+K_4)_{p_2})\\
&+2K_7(I_4^2-\frac{1}{4}K_4^2+J_2+I_2I_3+f^2_{u_2}+\frac{1}{2}D_x(2I_4+K_4))+\frac{1}{2}(2I_4+K_4)_{u_1}-I_{3u_2}\\
&-2K_5+I_2I_{3p_1},\\
&L_6=\frac{1}{J_2J_3}\bigg(2K_5(I_2J_3^2+I_3J_2^2)+J_3(J_3J_{2u_2}-J_2J_{3u_2})+J_3(J_2^2I_{3p_2}-J_3^2I_{2p_2})\bigg)\\
&+J_3(I_1-I_4)_{p_2}+2K_5(I_1+I_4+2K_4)-2D_xK_5,\\
&L_7=2J_2K_{5p_1}-2J_3K_{5p_2}+8K_5^2,\\
&L_8=-4K_5K_7-2K_{5p_1},\\
&L_9=-\frac{1}{4J_2^2J_3}\bigg(8J_2^3I_3K_7-8J_2^2I_3K_5+4I_2J_3^2K_5+J_2J_3(8K_5(2I_4+K_4)-4I_2J_3K_7)\\
&+4J_2J_3^2(I_1-I_4)_{p_2}+2J_3(J_3J_{2u_2}-2J_2J_{3u_2})-2J_3^3I_{2p_2}\bigg)-\frac{1}{4}\bigg(2K_7(4I_1-3K_4-2I_4)+4I_{3p_2}\\
&+(2I_4+K_4)_{p_1}-8D_xK_7\bigg),\\
&L_{10}=4K_5K_7-2J_2K_{7p_1}+2J_3K_{7p_2},\\
&L_{11}=2K_{7p_1}-4K_7^2,\\
&L_{12}=\frac{1}{J_2^2J_3}\bigg(J_2J_3^2(4K_5(I_1-I_4)_{p_2}-8I_2K_5K_7-2K_7J_{2u_2})+8J_2J_3K_5^2(2I_4+K_4)\\
&+2J_3K_5(J_3J_{2u_2}-2J_2J_{3u_2})-2K_5(2J_2^3I_{3p_1}+J_3^3I_{2p_2})+4K_5^2(I_2J_3^2-I_3J_2^2)\\
&+2J_2K_7(J_3^3I_{2p_2}+J_2^3I_{3p_1})\bigg)+J_2((2I_1+K_4)K_{7p_1}+2(I_3K_{7p_2}-K_7I_{3p_2})+2K_7(I_1-I_4)_{p_1}\\
&-2K_{7u_1})+J_3(2K_7(I_4-I_1)_{p_2}+2K_{7u_2}-2(I_2K_7)_{p_1}-(2I_4+K_4)K_{7p_2})+2K_5(I_4-I_1)_{p_1}\\
&-4K_5K_7(2I_4+3K_4)-(2I_1+K_4)K_{5p_1}+2(3K_5I_{3p_2}-I_3K_{5p_2})+2K_7J_{3u_2}+2K_{5u_1}\\
&+4(K_7D_xK_5-K_5D_xK_7).\\
\end{aligned}
\end{equation}

\end{document}